\documentclass[11pt]{amsart}

\usepackage{amsmath, amssymb, amsthm}
\usepackage{mathtools}
\usepackage{tikz-cd}
\usepackage{hyperref}
\usepackage{cleveref}
\usepackage{enumitem}
\usepackage[margin=1in]{geometry}

\newtheorem{maintheorem}{Theorem}

\theoremstyle{plain}
\newtheorem{theorem}{Theorem}[section]
\newtheorem{proposition}[theorem]{Proposition}
\newtheorem{lemma}[theorem]{Lemma}
\newtheorem{corollary}[theorem]{Corollary}

\theoremstyle{definition}
\newtheorem{definition}[theorem]{Definition}
\newtheorem{example}[theorem]{Example}
\newtheorem{remark}[theorem]{Remark}
\newtheorem{convention}[theorem]{Convention}

\newcommand{\Hom}{\mathrm{Hom}}
\newcommand{\Set}{\mathbf{Set}}
\newcommand{\Ab}{\mathbf{Ab}}
\newcommand{\CRing}{\mathbf{CRing}}
\newcommand{\CMonoid}{\mathbf{CMonoid}}
\newcommand{\CSemiring}{\mathbf{CSemiring}}
\newcommand{\GSet}{\Gamma\mathbf{Set}_*}
\newcommand{\FMod}{\mathbb{F}_1\mathbf{Mod}}
\newcommand{\FAlg}{\mathbb{F}_1\mathbf{Alg}}
\newcommand{\FF}{\mathbb{F}_1}
\newcommand{\ZZ}{\mathbb{Z}}
\newcommand{\NN}{\mathbb{N}}
\newcommand{\RR}{\mathbb{R}}

\hypersetup{
    colorlinks=true,
    linkcolor=blue,
    citecolor=blue,
    urlcolor=blue
}

\begin{document}

\title{Hyper-Operations and Coherent Plasmas over $\mathbb{F}_1$}

\author{Luqiao Xu}
\address{Universitat Autònoma de Barcelona}
\email{Luqiao.Xu@uab.cat}

\begin{abstract}
The additive structure of $\mathbb{F}_1$-modules (in the sense of Segal's $\Gamma$-sets) differs fundamentally from that of abelian groups: addition is encoded through a family of $n$-ary hyper-operations that are multivalued and need not satisfy classical associativity. We establish a \emph{law of generalized associativity} showing that, despite this failure of strict associativity, every $n$-ary sum is contained in the
iterated hypersum associated with any partition
of its inputs. This enables us to give a hyper-additive description
on $\pi_0$ of Schwede's linearization functor $-\otimes_{\mathbb{F}_1} \mathbb{Z}: \mathbb{F}_1\mathbf{Mod} \to \mathbf{Ab}$ that universally strictifies the hyper-additive structure of $\mathbb{F}_1$-modules into classical abelian group addition. We then introduce coherent plasmas, whose algebraic structure encodes the coherence condition inherent in $\mathbb{F}_1$-modules. We construct fully faithful nerves for these structures
and their finite-arity truncations, yielding a compatible
tower of adjunctions, as well as a coherent Segal condition that characterizes the essential image. The levels of these nerves are represented by
powerset plasmas $P_n$, which arise as algebraic
realizations of representable $\Gamma$-sets,
analogous to standard simplices in geometric realization.
\end{abstract}
\maketitle


\tableofcontents

\section{Introduction}

Multiplicative monoids, semirings, and hyperrings exhibit different forms of additive structure: addition may be unspecified, single-valued, or multivalued. Segal's $\Gamma$-sets provide a setting in which
these forms of addition can be studied through
functorial data.
Following Connes and Consani
\cite{connes2019absolute}, we regard
$\Gamma$-sets as $\mathbb F_1$-modules.

For a $\Gamma$-set $X$, an element of $X(n_+)$
determines $n$ coordinates in $X(1_+)$ and a
total sum.
Collecting the possible total sums for fixed
coordinates defines an $n$-ary hyper-operation
on $X(1_+)$, whose values may be empty or contain
more than one element.
For Eilenberg--MacLane objects these operations
recover ordinary addition.
In general, however, an $n$-ary hypersum need
not agree with an iterated binary hypersum,
and binary hyperaddition need not be associative.

We study two aspects of this additive structure.
First, we determine how its relations present
the universal map to an abelian group.
Second, we formulate an intrinsic compatibility
condition under which a family of hyper-operations
can be recovered from a canonical $\Gamma$-set.
The first question leads to a description of
linearization using only binary relations.
The second leads to coherent plasmas and
a Segal-type coherence condition.

\subsection*{Main results}

Our first result describes the behavior of
hypersums under partitions of their inputs.
As usual, hyper-operations are extended to
subsets by taking unions.

\begin{maintheorem}[Generalized associativity]
\label{thm:main-A}
Let $X$ be a $\Gamma$-set and let
$I=\{1,\ldots,n\}$.
For every partition
$I=I_1\amalg\cdots\amalg I_m$
and subsets $A_i\subseteq X(1_+)$, we have
\[
\bigoplus_{i\in I}A_i
\subseteq
\bigoplus_{j=1}^m
\left(\bigoplus_{i\in I_j}A_i\right).
\]
\end{maintheorem}

In particular, every $n$-ary hypersum is
contained in every iterated binary hypersum
with the same inputs, and
$$ a \oplus b \oplus c \subseteq  (a \oplus b) \oplus c $$
The inclusion can be strict:
binary iteration may admit intermediate choices
that do not arise from a single element of
$X(n_+)$.
Nevertheless, the inclusion implies that
a pointed map to a commutative monoid which
preserves binary hyperaddition also preserves
all higher-arity sums.

This yields an explicit presentation of
extension of scalars to $\ZZ$:

\begin{maintheorem}[Extension of scalars]
\label{thm:main-B}
For an $\FF$-module $X$, set
\[
X\otimes_{\FF}\ZZ
=
\ZZ[X(1_+)]/
\left\langle
[*],\ [a]+[b]-[c]\ \middle|\ c\in a\oplus b
\right\rangle.
\]
This construction defines a left adjoint
to the Eilenberg--MacLane functor:
\[
-\otimes_{\FF}\ZZ:\FMod
\rightleftarrows\Ab:H.
\]
For a commutative $\FF$-algebra $A$,
the formulas $[a][b]=[ab]$ and $1=[1_A]$
give $A\otimes_{\FF}\ZZ$ a natural
commutative ring structure and yield
an adjunction
\[
-\otimes_{\FF}\ZZ:\FAlg
\rightleftarrows\CRing:H.
\]
\end{maintheorem}

The underlying module adjunction is $\pi_0$ of Schwede's linearization
adjunction restricted to discrete $\Gamma$-sets
\cite{schwede}.
We link it with generalized associativity, providing the explanation of why binary hyperaddition relations suffice for every arity.

The construction recovers familiar algebraic
operations.
On spherical objects it gives free abelian
groups, and on Eilenberg--MacLane objects
of commutative monoids it gives group completion.
For spherical monoid algebras, it recovers
the contracted monoid ring construction underlying
Deitmar's base change
\cite{deitmar2006schemesf1}.

Our second question concerns the simultaneous
compatibility of hyper-operations. The hyper-operations of a $\Gamma$-set satisfy more than
the partitionwise inclusions of
generalized associativity.
An element $\xi\in X(k_+)$ determines not only
its coordinates and total sum, but also a
partial sum for every subset of $[k]$.
These partial sums are compatible with all
groupings into pairwise disjoint subsets.
We formulate this property intrinsically
as an algebraic structure.

A \emph{coherent plasma}, or \emph{$\infty$-plasma}, is a pointed set
equipped with hyper-operations
of all finite arities such that every $n$-ary hypersum extends to a family of compatible subset sums.
Truncating both the operations and the
compatibility conditions to arities at most $n$
defines the category $\mathsf{Plas}_n$.
For $n=2$, this recovers the category
of weakly unital commutative hypermagmas,
called plasmas by Beardsley and Nakamura
\cite{beardsley2024}.

Let $\mathcal G_n$ denote the category of
$n$-truncated $\Gamma$-sets, with
$\mathcal G_\infty=\FMod$, and write
$\mathsf{Plas}_\infty$ for coherent plasmas.

\begin{maintheorem}[Coherent nerves and truncation]
\label{thm:main-C}
For every $1\leq n\leq\infty$, extraction
of hyper-operations defines an adjunction
\[
T_n:\mathcal G_n
\rightleftarrows\mathsf{Plas}_n:N_n,
\qquad T_n\dashv N_n,
\]
and $N_n$ is fully faithful, identifing
$\mathsf{Plas}_n$ with a reflective full
subcategory of $\mathcal G_n$.

In particular, the nerve of a coherent plasma $M$ is
the $\Gamma$-set defined by
\[
N_\infty M
=
\Hom_{\mathsf{Plas}_\infty}(P(-),M),
\]
where $P_n$ is the powerset plasma.
\end{maintheorem}

The construction is analogous to the classical nerve of a category, whose \(n\)-simplices are functors from the ordinal category \([n]\). \(P_n\) serves as a standard additive test object: maps \(P_n\to M\) encode compatible systems of partial sums of \(n\) elements. Whereas the categorical nerve encodes ordered composition of at most $n$ morphisms, the plasmic nerve encodes symmetric addition with at most $n$ summands, with finite pointed maps organizing permutations, grouping, and omission of summands.

We also characterize the essential image of
$N_\infty$ by a \emph{coherent Segal condition}.
For $S\subseteq[n]$, let $\rho_S:n_+\to1_+$
take the value $1$ exactly on $S$.
The subset-sum map
\[
\sigma_n:X(n_+)\longrightarrow
\prod_{S\subseteq[n]}X(1_+),
\qquad
\xi\longmapsto
\bigl(X(\rho_S)(\xi)\bigr)_S
\]
identifies $X(n_+)$ with the coherent family
 precisely when $X$ belongs to this
essential image.
The ordinary Segal condition implies this condition.

\subsection*{Organization}

Section~\ref{sec:preliminaries} recalls
$\Gamma$-sets and their monoidal structure.
Section~\ref{sec:hyper-ops} introduces their
hyper-operations and gives examples of
nonassociative hyperaddition.
Section~\ref{subsec:gen-assoc} proves
generalized associativity.
Section~\ref{sec:extension-modules} constructs
extension of scalars for modules and algebras.
Finally, Section~\ref{sec:coherent-plasmas}
develops coherent plasmas, their nerves,
and the finite-arity tower.

\subsection*{Acknowledgments}

The author thanks Caterina Consani for guidance
and Jonathan Beardsley for useful conversations. The author also thanks Xavier Xarles Ribas for comments as well as Lean verifications.
This paper is based on part of the author's
Ph.D. thesis \cite{xu-thesis}.

\section{Preliminaries: $\Gamma$-sets}\label{sec:preliminaries}

In this section, we review the theory of Segal's $\Gamma$-sets following \cite{dundas2012local} and \cite{connes2019absolute}. We establish the basic functors and monoidal structure that will be used throughout the paper. The basic idea is that $\Gamma$-sets encode additive structure through their functoriality from the category $\Gamma^{\mathrm{op}}$ of pointed finite sets, while additional multiplicative structure can be imposed to obtain $\mathbb{F}_1$-algebras.

\subsection{The category $\Gamma^{\mathrm{op}}$}

\begin{definition}\label{def:gamma-op}
The category $\Gamma^{\mathrm{op}}$ has:
\begin{itemize}
\item Objects: Pointed finite sets $n_+ := \{0, 1, \ldots, n\}$ for each integer $n \geq 0$, where $0$ is the basepoint.
\item Morphisms: Pointed maps (i.e., maps preserving the basepoint $0$).
\end{itemize}
Equivalently, $\Gamma^{\mathrm{op}}$ is the skeleton of the category of finite pointed sets.
\end{definition}

\begin{definition}\label{def:gamma-set}
A \emph{$\Gamma$-set} is a functor $F: \Gamma^{\mathrm{op}} \to \Set_*$ from $\Gamma^{\mathrm{op}}$ to the category of pointed sets such that $F(0_+)=\{*\}$. The pointed set $F(n_+)$ is called the \emph{$n$-th level} of $F$. We denote by $\GSet$ the category of $\Gamma$-sets, whose morphisms are natural transformations.
\end{definition}

\begin{remark}
In the language of Segal \cite{segal1974categories}, $\Gamma$-sets are $\Gamma$-spaces with discrete topology, where $\Gamma$-spaces provide a model for connective spectra. In the language of Connes-Consani \cite{connes2019absolute}, $\Gamma$-sets are $\mathbb{F}_1$-modules (or denoted as $\mathbb{S}$-modules), emphasizing their role as the fundamental objects in $\mathbb{F}_1$-geometry. This connection to homotopy theory is one reason $\Gamma$-sets provide a natural framework for absolute algebraic geometry.
\end{remark}

\subsection{The Eilenberg-MacLane functor}

The following construction shows how commutative monoids embed into $\Gamma$-sets, revealing that classical additive structures are special cases of the more general hyper-additive structures we will study.

\begin{definition}\label{def:eilenberg-maclane}
Let $M$ be a commutative monoid written additively with identity $0$. The \emph{Eilenberg-MacLane object} $HM$ is the $\Gamma$-set defined by:
$$HM(n_+) := M^{\oplus n} = \{(m_1, \ldots, m_n) : m_i \in M\}$$
with basepoint $(0, 0, \ldots, 0)$. For a morphism $f \in \Hom_{\Gamma^{\mathrm{op}}}(n_+, k_+)$, the induced map is:
$$HM(f): M^{\oplus n} \to M^{\oplus k}, \quad (m_1, \ldots, m_n) \mapsto \left(\sum_{f(i)=1} m_i, \ldots, \sum_{f(i)=k} m_i\right).$$
\end{definition}

To understand this construction, we examine the key morphisms in $\Gamma^{\mathrm{op}}$ that encode additive operations:

\begin{example}\label{ex:EM-morphisms}
Consider the following morphisms in $\Gamma^{\mathrm{op}}$ at level $n=2$:

\begin{enumerate}[label=(\arabic*)]
\item The \emph{sum morphism} $s \in \Hom_{\Gamma^{\mathrm{op}}}(2_+, 1_+)$ is defined by $s(0) = 0$, $s(1) = 1$, $s(2) = 1$. 

For an Eilenberg-MacLane object $HM$, this induces:
$$HM(s): M \times M \to M, \quad (a,b) \mapsto a+b.$$

This is the usual addition in $M$.

\item The \emph{first projection} $p_1 \in \Hom_{\Gamma^{\mathrm{op}}}(2_+, 1_+)$ is defined by $p_1(0) = 0$, $p_1(1) = 1$, $p_1(2) = 0$.

This induces $HM(p_1)(a,b) = a$.

\item The \emph{second projection} $p_2 \in \Hom_{\Gamma^{\mathrm{op}}}(2_+, 1_+)$ is defined by $p_2(0) = 0$, $p_2(1) = 0$, $p_2(2) = 1$.

This induces $HM(p_2)(a,b) = b$.
\end{enumerate}

The identities $s\circ\iota_1=s\circ\iota_2=\mathrm{id}_{1_+}$ in $\Gamma^{\mathrm{op}}$ encode the unit axiom $a + 0 = 0 + a = a$ for the monoid $M$, where $\iota_j:1_+\to2_+$ are the pointed maps defined by
$\iota_j(1)=j$, for $j=1,2$. More generally, for any $n$, the morphisms $p_{i,n}$ (projection onto the $i$-th coordinate) and $s_n$ (sum of all coordinates) will play a central role in defining $n$-ary sums in Section \ref{sec:hyper-ops}.
\end{example}

\begin{proposition}[\cite{connes2019absolute}, Proposition 2.3]\label{prop:EM-functor}
The assignment $M \mapsto HM$ defines a fully faithful functor
$$H: \CMonoid \to \GSet$$
from commutative monoids to $\Gamma$-sets.
\end{proposition}

The Eilenberg-MacLane construction reveals how commutative monoids encode their additive structure in a $\Gamma$-set: the $n$-th level $HM(n_+)$ consists of all $n$-tuples of elements from $M$, and morphisms describe how to ``sum'' these tuples according to the arrows of the morphism.

\subsection{The spherical functor}

While the Eilenberg-MacLane functor embeds algebraic structures such as commutative monoids into $\Gamma$-sets, we also need a ``free'' construction that builds $\Gamma$-sets from pointed sets. This is provided by the spherical functor, which produces $\Gamma$-sets with minimal additive structure.

\begin{definition}\label{def:spherical}
Let $X$ be a pointed set with basepoint $*$. The \emph{spherical  object} $\mathbb{S}X$ is the $\Gamma$-set defined by:
$$\mathbb{S}X(n_+) := X \wedge n_+$$
where $\wedge$ denotes the smash product of pointed sets. For a morphism $f \in \Hom_{\Gamma^{\mathrm{op}}}(n_+, k_+)$, we define:
$$\mathbb{S}X(f) := \mathrm{id}_X \wedge f: X \wedge n_+ \to X \wedge k_+.$$
\end{definition}

\begin{proposition}[\cite{connes2019absolute}, Proposition 2.2]\label{prop:spherical-functor}
The assignment $X \mapsto \mathbb{S}X$ defines a fully faithful functor
$$\mathbb{S}: \Set_* \to \GSet$$
that is left adjoint to the evaluation functor $-(1_+): \GSet \to \Set_*$ sending a $\Gamma$-set to its first level.
\end{proposition}

\begin{proof}
For a morphism $f: X \to Y$ of pointed sets, we define $\mathbb{S}f: \mathbb{S}X \to \mathbb{S}Y$ levelwise by $f \wedge \mathrm{id}_{n_+}$. The adjunction
$$\Hom_{\GSet}(\mathbb{S}X, F) \cong \Hom_{\Set_*}(X, F(1_+))$$
follows from the universal property of the smash product and the Yoneda lemma.
\end{proof}

\begin{example}\label{ex:deff1}
The $\Gamma$-set $\FF := \mathbb{S}\{0,1\}$ satisfies $\FF(n_+) = n_+$ for all $n$, and morphisms act by composition. This is the \emph{field with one element}, serving as the monoidal unit in $\GSet$ (see Theorem \ref{thm:monoidal} below).
\end{example}

\subsection{The monoidal structure}\label{subsec:monoidal}

The category $\GSet$ inherits a symmetric monoidal structure from the smash product on $\Set_*$ via Day convolution. 

\begin{theorem}[\cite{connes2019absolute}, \cite{dundas2012local}]\label{thm:monoidal}
The category $\GSet$ is a symmetric monoidal closed category with:
\begin{itemize}
\item Monoidal product: The smash product $\wedge$ defined by
$$(F \wedge G)(k_+) := \mathrm{colim}_{m_+ \wedge n_+ \to k_+} F(m_+) \wedge G(n_+)$$
where the colimit is taken over all morphisms $m_+ \wedge n_+ \to k_+$ in $\Gamma^{\mathrm{op}}$.

\item Monoidal unit: $\FF = \mathbb{S}\{0,1\}$.

\item Internal Hom: The category is closed, meaning for $\Gamma$-sets $M$ and $N$, there exists an internal hom object $\underline{\Hom}_{\GSet}(M,N)$ defined by
$$\underline{\Hom}_{\GSet}(M,N)(n_+) := \Hom_{\GSet}(M, N_{n_+ \wedge -})$$
where $N_{n_+ \wedge -}$ is the $\Gamma$-set given by $(N_{n_+ \wedge -})(k_+) := N(n_+ \wedge k_+)$.
\end{itemize}
\end{theorem}

This monoidal structure allows us to define $\mathbb{F}_1$-algebras as monoid objects in $(\GSet, \wedge, \FF)$, analogous to how rings are monoid objects in $(\Ab, \otimes_\ZZ, \ZZ)$. We will develop this in Section \ref{sec:algebras}.

\begin{convention}
Following Connes-Consani, we often refer to $\Gamma$-sets as \emph{$\mathbb{F}_1$-modules}, emphasizing their role in $\mathbb{F}_1$-geometry. The notations $\GSet$ and $\FMod$ are used interchangeably: just as $\ZZ$-modules are abelian groups, $\mathbb{F}_1$-modules are $\Gamma$-sets.
\end{convention}

Having established the categorical framework of $\Gamma$-sets, we now turn to investigating the additive structure they encode.

\section{Hyper-operations in $\mathbb{F}_1$-modules}\label{sec:hyper-ops}

In this section, we analyze the additive structure encoded in $\Gamma$-sets. Unlike classical abelian groups, where addition is a binary operation $+: M \times M \to M$ with a unique output, the additive data of a $\Gamma$-set is encoded through a family of $n$-ary \emph{hyper-operations}---operations whose outputs are sets rather than single elements. We formalize this notion through the concept of $n$-ary sums, establish that strict associativity fails in general, and prove the law of generalized associativity.

\subsection{Hyper-operations: Definition and basic properties}

Before defining $n$-ary sums in $\Gamma$-sets, we establish the general framework of hyper-operations from hyperstructure theory. This provides the language for discussing multivalued operations systematically.

\begin{definition}\label{def:hyper-op}
An \emph{$n$-ary hyper-operation} on a nonempty set $H$ is a map
$$f: H^n \to \mathcal{P}(H)$$
where $\mathcal{P}(H)$ denotes the power set of $H$. The number $n$ is called the \emph{arity} of $f$.

In particular, a \emph{binary hyper-operation} is a map
$$\oplus: H \times H \to \mathcal{P}(H).$$
\end{definition}

\begin{remark}
In some literature \cite{hyper}, hyper-operations are restricted to $\mathcal{P}^*(H)$, the non-empty subsets of $H$, ensuring outputs are always ``multivalued'' in a strict sense. We allow the empty set as an output to accommodate cases where certain operations are undefined, as it occurs naturally in $\mathbb{F}_1$-modules (see Example \ref{ex:f1-sums} below). An algebraic structure satisfying similar axioms to a field, but where the addition is a hyper-operation is called a \emph{hyperfield}. \cite{hyper}
\end{remark}

\begin{example}[Krasner hyperfield]\label{ex:krasner}
The two-element set $\mathbb{K} = \{0,1\}$ with multiplication $0 \cdot x = 0$, $1 \cdot 1 = 1$, and a binary hyper-operation $\oplus$:
\begin{center}
\begin{tabular}{c|cc}
$\oplus$ & $0$ & $1$ \\
\hline
$0$ & $0$ & $1$ \\
$1$ & $1$ & $\{0,1\}$
\end{tabular}
\end{center}
forms a \emph{hyperfield} arising as the quotient $\RR/\RR^\times$ \cite{krasner}. The hyper-sum $1 + 1 = \{0,1\}$ reflects the fact that the sum of two nonzero reals can be either zero (if they cancel) or nonzero.
\end{example}

\begin{example}[Sign hyperfield]\label{ex:sign}
The sign hyperfield extends this to capture signs of real numbers:
\begin{center}
\begin{tabular}{c|ccc}
$\oplus$ & $0$ & $1$ & $-1$ \\
\hline
$0$ & $0$ & $1$ & $-1$ \\
$1$ & $1$ & $1$ & $\{-1,0,1\}$ \\
$-1$ & $-1$ & $\{-1,0,1\}$ & $-1$
\end{tabular}
\end{center}
arising as $\RR/\RR_+^\times$.
\end{example}

These examples suggest that $\FF$ itself could be interpreted as a ``hyperfield'' with an even more degenerate additive structure:

\begin{example}[The ``hyperfield with one element'']\label{ex:f1}
We propose $\mathbb{F} = \{0,1\}$ with usual multiplication and hyper-addition:
\begin{center}
\begin{tabular}{c|cc}
$+$ & $0$ & $1$ \\
\hline
$0$ & $0$ & $1$ \\
$1$ & $1$ & $\emptyset$
\end{tabular}
\end{center}
The empty set $1 + 1 = \emptyset$ reflects that there is ``no way'' to add two nonzero elements---the structure is purely multiplicative. Note that in classic terms, this is not a hyperfield since $1$ does not have an additive inverse. We will discuss the relation of the ``hyperfield with one element'' and $\mathbb{F}_1$ in Example \ref{ex:f1-sums}.
\end{example}

A binary hyper-operation naturally extends to operate on subsets:

\begin{convention}\label{conv:subset-extension}
Given a binary hyper-operation $\oplus: H \times H \to \mathcal{P}(H)$ and subsets $A, B \subseteq H$, we define
$$A \oplus B := \bigcup_{a \in A, b \in B} (a \oplus b).$$
We also write $a \oplus B := \{a\} \oplus B$ and $A \oplus b := A \oplus \{b\}$, identifying elements with their singleton sets when convenient.
\end{convention}

\subsection{The $n$-ary sum in $\Gamma$-sets}

We now define the notion of $n$-ary sum in a $\Gamma$-set, which encodes the $n$-ary hyper-additive structure \cite{connes2011hyperring}. The central idea is that elements of $X(n_+)$ ``exhibits''  $n$-ary sums: an element $z \in X(n_+)$ whose projections onto level $1$ are $a_1, \ldots, a_n$ exhibits their sum as $X(s_n)(z)$.

One can think of $X(n_+)$ as the space of ``$n$-tuples''. An element $z \in X(n_+)$ has $n$ ``coordinates'' extracted by projection morphisms $p_1, \ldots, p_n$, and a ``sum'' extracted by the sum morphism $s_n$. The $n$-sum $\bigoplus_{i=1}^n a_i$ collects all possible sums that can arise from elements whose coordinates are $a_1, \ldots, a_n$.

To make this precise, we identify the key morphisms in $\Gamma^{\mathrm{op}}$:

\begin{itemize}
\item The \emph{projection onto the $i$-th coordinate} $p_{i,n} \in \Hom_{\Gamma^{\mathrm{op}}}(n_+, 1_+)$ is defined by
$$p_{i,n}(j) = \begin{cases} 1 & \text{if } j = i \\ 0 & \text{otherwise} \end{cases}$$

Diagrammatically, $p_{i,n}$ sends:
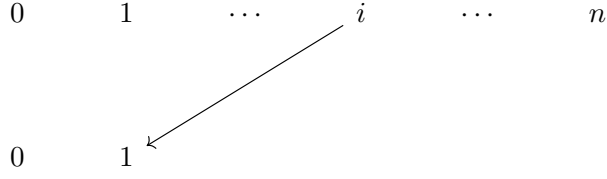
\begin{center}
\begin{tikzcd}
0 & 1 & \cdots & i \arrow[ddll] & \cdots & n \\
\\
0 & 1
\end{tikzcd}
\end{center}

where unlabeled arrows map to the basepoint $0$.

\item The \emph{sum morphism} $s_n \in \Hom_{\Gamma^{\mathrm{op}}}(n_+, 1_+)$ is defined by
$$s_n(j) = \begin{cases} 1 & \text{if } j \neq 0 \\ 0 & \text{if } j = 0 \end{cases}$$

Diagrammatically:
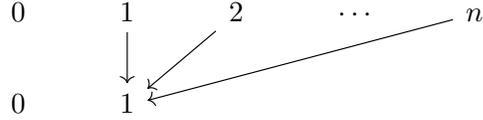
\begin{center}
\begin{tikzcd}
0 & 1 \arrow[d] & 2 \arrow[dl] & \cdots & n \arrow[dlll] \\
0 & 1
\end{tikzcd}
\end{center}
\end{itemize}

\begin{definition}\label{def:n-sum}
Let $X$ be a $\Gamma$-set, and let $\{a_i\}_{i=1}^n$ be a finite collection of elements in $X(1_+)$. We define their \emph{$n$-ary sum}, denoted $\bigoplus_{i=1}^n a_i$ or simply $\bigoplus a_i$, as the subset of $X(1_+)$ given by
$$\bigoplus_{i=1}^n a_i := \left\{X(s_n)(z) \,\bigg|\, z \in X(n_+) \text{ and } X(p_{i,n})(z) = a_i \text{ for all } i = 1, \ldots, n\right\}.$$

An element $z \in X(n_+)$ satisfying $X(p_{i,n})(z) = a_i$ for all $i$ is said to \emph{exhibit} the sum $X(s_n)(z) \in \bigoplus a_i$.
\end{definition}

\begin{convention}\label{conv:special-sums}
By convention, we define:
\begin{itemize}
\item The \emph{$1$-ary sum} of a single element $a \in X(1_+)$ is $\bigoplus_{i=1}^1 a_i = \{a\}$.
\item The \emph{$0$-ary sum} (empty sum) is $\bigoplus_{i=1}^0 a_i = \{*\}$, where $*$ is the basepoint.
\end{itemize}
These are consistent with Definition \ref{def:n-sum} since $p_{1,1}$ and $s_1$ are both the identity morphism, and $X(0_+) = \{*\}$ has only the basepoint. When the level $n$ is clear from context, we write $p_i$ instead of $p_{i,n}$ and $s$ instead of $s_n$ to simplify notation.
\end{convention}

For subsets $A_1, \ldots, A_n \subseteq X(1_+)$, we extend the definition by
$$\bigoplus_{i=1}^n A_i := \bigcup_{a_i \in A_i} \bigoplus_{i=1}^n a_i.$$

For the binary sum, we write $a \oplus b$ for brevity.

\subsection{Examples}

We now present several examples that illustrate the behavior of $n$-ary sums and reveal the spectrum from classical strict addition to maximally degenerate hyper-addition.

\begin{example}[Eilenberg-MacLane objects]\label{ex:EM-sums}
Let $M$ be a commutative monoid and consider $HM$ from Definition \ref{def:eilenberg-maclane}. For any collection $\{a_i\}_{i=1}^n \subseteq M = HM(1_+)$, there is a \emph{unique} element $z \in HM(n_+)$ exhibiting their sum, namely the $n$-tuple $z = (a_1, a_2, \ldots, a_n)$.

Therefore,
$$\bigoplus_{i=1}^n a_i = \left\{\sum_{i=1}^n a_i\right\},$$
a singleton set containing the usual sum in $M$. This shows that our definition recovers classical addition when the $\Gamma$-set has the form $HM$. In this sense, Eilenberg-MacLane objects represent the case where hyper-addition is ``strict''---there is exactly one way to sum any collection of elements.
\end{example}

\begin{example}[The field with one element]\label{ex:f1-sums}
For $\FF = \mathbb{S}\{0,1\}$ from Example \ref{ex:deff1}, we have $\FF(n_+) = n_+ = \{0,1,\ldots,n\}$. Then
$$1 \oplus 1 = \emptyset.$$

Indeed:
\begin{itemize}
\item For $1 \oplus 0$: the element $z = 1 \in \FF(2_+)$ satisfies $p_1(1) = 1$, $p_2(1) = 0$, and $s(1) = 1$, so $1 \oplus 0 = \{1\}$.
\item For $0 \oplus 1$: the element $z = 2 \in \FF(2_+)$ satisfies $p_1(2) = 0$, $p_2(2) = 1$, and $s(2) = 1$, so $0 \oplus 1 = \{1\}$.
\item For $0 \oplus 0$: the element $z = 0 \in \FF(2_+)$ satisfies $p_1(0) = 0$, $p_2(0) = 0$, and $s(0) = 0$, so $0 \oplus 0 = \{0\}$.
\end{itemize}

Yet there is no ``fourth'' element in $\FF(2_+)$ that exhibits $1 \oplus 1$. This reflects the purely multiplicative nature of $\FF$: there is ``no way'' to add two nonzero elements, justifying the term ``hyperfield with one element'' from Example \ref{ex:f1}.
\end{example}

\begin{example}[Spherical objects]\label{ex:spherical-sums}
More generally, let $X$ be a pointed set with at least two elements and consider $\mathbb{S}X$ from Definition \ref{def:spherical}. 

In this case, $\mathbb{S}X(s_n)(x,j) = x$, so
\[
\bigoplus_{i=1}^n a_i=
\begin{cases}
\{*\}, & \text{if all } a_i=*,\\
\{a_j\}, & \text{if exactly one } a_j\ne *,\\
\varnothing, & \text{if at least two inputs are nonzero}.
\end{cases}
\]

In particular, for nonzero $a, b \in X$, we have $a \oplus b = \emptyset$, generalizing Example \ref{ex:f1-sums}. The only nontrivial sums are those involving the basepoint: $a \oplus * = \{a\}$ and $* \oplus * = \{*\}$. This reflects the purely multiplicative nature of spherical objects: they have no additive rules except those forced by the basepoint axioms.
\end{example}

The following lemma shows that morphisms of $\Gamma$-sets respect $n$-ary sums, though not necessarily strictly (the image of a sum may be contained in, but not equal to, the sum of images):

\begin{lemma}\label{lem:preserve-sums}
Let $f: X \to Y$ be a morphism of $\Gamma$-sets, and let $\{A_i\}_{i=1}^n$ be a finite collection of subsets of $X(1_+)$. Then
$$f\left(\bigoplus_{i=1}^n A_i\right) \subseteq \bigoplus_{i=1}^n f(A_i).$$

In particular, for elements $a, b \in X(1_+)$, we have $f(a \oplus b) \subseteq f(a) \oplus f(b)$.
\end{lemma}

\begin{proof}
It suffices to show that if $c \in \bigoplus_{i=1}^n a_i$ for $a_i \in A_i$, then $f(c) \in \bigoplus_{i=1}^n f(a_i)$.

Suppose $c \in \bigoplus_{i=1}^n a_i$. Then there exists $z \in X(n_+)$ exhibiting $c$ as the sum of $\{a_i\}$:
$$X(p_i)(z) = a_i \quad \text{for all } i, \quad \text{and} \quad X(s_n)(z) = c.$$

Consider $f(z) \in Y(n_+)$. By functoriality of $f$ (i.e., naturality of the natural transformation), we have:
$$Y(p_i)(f(z)) = f(X(p_i)(z)) = f(a_i)$$
and
$$Y(s_n)(f(z)) = f(X(s_n)(z)) = f(c).$$

Therefore, $f(z)$ exhibits $f(c)$ as the sum of $\{f(a_i)\}$, which means $f(c) \in \bigoplus_{i=1}^n f(a_i)$.
\end{proof}

\begin{remark}
The inclusion in Lemma \ref{lem:preserve-sums} can be strict. For example, consider a morphism $f: \mathbb{S}M \to HR$ from a spherical monoid algebra to an Eilenberg-MacLane algebra, and take nonzero $a, b \in M$. Then $a \oplus b = \emptyset$ in $\mathbb{S}M$ by Example \ref{ex:spherical-sums}, so
$$f(a \oplus b) = f(\emptyset) = \emptyset \subsetneq \{f(a) +_R f(b)\} = f(a) \oplus f(b).$$
This shows that morphisms preserve hyper-additive structure, but do not necessarily reflect it: the target may have more sums than the source.
\end{remark}

\subsection{The failure of strict associativity}\label{subsec:failure}

Now we consider the natural question: can $n$-ary sums be computed using iterated binary sums? In classical algebra, associativity ensures that
$$a_1 + a_2 + a_3 = (a_1 + a_2) + a_3 = a_1 + (a_2 + a_3).$$

For $\Gamma$-sets, we have three potentially different sets:
\begin{itemize}
\item The ternary sum: $a_1 \oplus a_2 \oplus a_3$
\item The left-associated binary sum: $(a_1 \oplus a_2) \oplus a_3$
\item The right-associated binary sum: $a_1 \oplus (a_2 \oplus a_3)$
\end{itemize}

One might hope these are always equal. However, this turns out to be false in general, as the next counter-example shows.

\begin{example}[Failure of strict associativity]\label{ex:failure-assoc}
Consider the quotient $\Gamma$-set $H\ZZ/H(3\ZZ)$, obtained by collapsing the sub-$\Gamma$-set $H(3\ZZ) \subseteq H\ZZ$ to the basepoint, the levelwise pointed-set quotient. Let $[n]$ denote the image of $n \in \ZZ$ in the quotient.

In $H\ZZ/H(3\ZZ)$, we compute the ternary sum:
$$[1] \oplus [2] \oplus [2] = \{[5]\}.$$
since in the quotient, $(1, 2, 2)$ is the unique element exhibiting the sum $1 + 2 + 2 = 5$.

On the other hand, consider the iterated binary sum:
\begin{align*}
([1] \oplus [2]) \oplus [2] &= [3] \oplus [2] \\
&= [0] \oplus [2] \quad \text{(since } [3] = [0] \text{ in the quotient)}\\
&= \{[3n + 2] : n \in \ZZ\}
\end{align*}

The equality $[0] \oplus [2]=\{[3n + 2] : n \in \ZZ\}$ holds because for any $n \in \ZZ$, image of the element $(3n, 2) \in H\ZZ (2_+)$ in the quotient exhibits the sum $[3n + 2]$.

Therefore:
$$[1] \oplus [2] \oplus [2] = \{[5]\} \subsetneq \{\ldots, [2], [5], [8], \ldots\} = ([1] \oplus [2]) \oplus [2].$$

The ternary sum is \emph{strictly contained in} the iterated binary sum, but they are not equal.

We also note that 
$$[1] \oplus ([2] \oplus [2]) = \{[5]\}.$$

Hence the law of associativity need not hold in a general $\FF$-module:
$$[1] \oplus ([2] \oplus [2]) \neq ([1] \oplus [2]) \oplus [2].$$
\end{example}

\section{The law of generalized associativity}\label{subsec:gen-assoc}

We now establish the first main result, which shows that while strict associativity fails, an inclusion property holds.

\begin{theorem}[Law of Generalized Associativity, Theorem \ref{thm:main-A}]\label{thm:gen-assoc}
Let $I = \{1, 2, \ldots, n\}$ be a finite index set, and let $\mathcal{P} = \{I_j\}_{j=1}^m$ be a partition of $I$ with $n_j := |I_j|$ and $n = \sum_{j=1}^m n_j$. Let $X$ be a $\Gamma$-set and let $\{A_i\}_{i=1}^n$ be a family of subsets of $X(1_+)$. Then
$$\bigoplus_{i \in I} A_i \subseteq \bigoplus_{j=1}^m \left(\bigoplus_{i \in I_j} A_i\right).$$
\end{theorem}

\begin{proof}
Let $a \in \bigoplus_{i \in I} A_i$. Then there exist elements $a_i \in A_i$ for all $i \in I$ such that $a \in \bigoplus_{i \in I} a_i$. By definition, there exists $z \in X(n_+)$ such that:
$$X(s_n)(z) = a \quad \text{and} \quad X(p_{i,n})(z) = a_i \quad \text{for all } i \in I.$$

Our goal is to show that $a \in \bigoplus_{j=1}^m \left(\bigoplus_{i \in I_j} a_i\right)$.

Let $f \in \Hom_{\Gamma^{\mathrm{op}}}(n_+, m_+)$ be the morphism determined by the partition $\mathcal{P}$, defined by:
$$f(i) = \begin{cases} j & \text{if } i \in I_j \\ 0 & \text{if } i = 0 \end{cases}$$

Set $y := X(f)(z) \in X(m_+)$. We claim that $y$ exhibits $a$ as the sum of the $m$ intermediate sums.

Indeed, we have:
$$X(s_m)(y) = X(s_m \circ f)(z) = X(s_n)(z) = a$$
where the second equality holds because $s_m \circ f = s_n$ (both morphisms send all nonzero elements of $n_+$ to $1 \in 1_+$).

Now we show $X(p_{j,m})(y) \in \bigoplus_{i \in I_j} a_i$ for each $j$. Fix $j \in \{1, \ldots, m\}$. We must construct an element $y_j \in X(n_j)$ that exhibits $X(p_{j,m})(y)$ as the sum of $\{a_i\}_{i \in I_j}$.

Let $\phi_j: I_j \to \{1, \ldots, n_j\}$ be an order-preserving bijection. Define $\tilde{\phi_j} \in \Hom_{\Gamma^{\mathrm{op}}}(n_+, (n_j)_+)$ by:
$$\tilde{\phi_j}(t) = \begin{cases} \phi_j(t) & \text{if } t \in I_j \\ 0 & \text{if } t \notin I_j \end{cases}$$

Set $y_j := X(\tilde{\phi_j})(z) \in X(n_j)$.

We verify that $y_j$ exhibits the desired sum. First:
$$X(s_{n_j})(y_j) = X(s_{n_j} \circ \tilde{\phi_j})(z).$$

Observe that both $s_{n_j} \circ \tilde{\phi_j}$ and $p_{j,m} \circ f$ send elements in $I_j$ to $1$ and all other elements to $0$. 

Therefore:
$$X(s_{n_j})(y_j) = X(s_{n_j} \circ \tilde{\phi_j})(z) = X(p_{j,m} \circ f)(z) = X(p_{j,m})(X(f)(z)) = X(p_{j,m})(y).$$

Second, for any $\ell \in \{1, \ldots, n_j\}$, set $i = \phi_j^{-1}(\ell) \in I_j$. Then:
$$X(p_{\ell,n_j})(y_j) = X(p_{\ell,n_j} \circ \tilde{\phi_j})(z) = X(p_{i,n})(z) = a_i \in A_i.$$

Therefore, $y_j$ exhibits $X(p_{j,m})(y)$ as the sum of $\{a_i\}_{i \in I_j}$, which means:
$$X(p_{j,m})(y) \in \bigoplus_{i \in I_j} a_i.$$

Since this holds for all $j \in \{1, \ldots, m\}$, we conclude that $y$ exhibits $a$ as the sum of the intermediate sums:
$$a = X(s_m)(y) \in \bigoplus_{j=1}^m \left(\bigoplus_{i \in I_j} a_i\right). \qedhere$$

\end{proof}

The following corollary makes explicit the case of iterated binary sums:

\begin{corollary}\label{cor:binary-iteration}
Let $X$ be a $\Gamma$-set, and let $\{a_i\}_{i=1}^n$ be a finite collection of elements in $X(1_+)$. Then
$$\bigoplus_{i=1}^n a_i \subseteq (a_1 \oplus (a_2 \oplus (\cdots \oplus (a_{n-1} \oplus a_n) \cdots))).$$

More generally, for any way of parenthesizing the expression $a_1 \oplus a_2 \oplus \cdots \oplus a_n$ using binary operations, the $n$-ary sum is contained in the result.
\end{corollary}

\begin{proof}
Apply Theorem \ref{thm:gen-assoc} repeatedly with appropriate binary partitions. For the specific right-nested parenthesization, use partitions $I = \{1\} \cup \{2, \ldots, n\}$, then $\{2\} \cup \{3, \ldots, n\}$, and so on.
\end{proof}

\begin{remark}\label{obs:interpretation}
    In this perspective, an $\mathbb{F}_1$-module $X$ can be viewed as consisting of (though with more information):
\begin{enumerate}[label=(\roman*)]
\item A pointed set $X(1_+)$ (the ``underlying set'')
\item A family of $n$-ary \emph{hyper-operations} $\bigoplus: X(1_+)^n \to \mathcal{P}(X(1_+))$ for each $n \geq 2$
\item Satisfying the law of \emph{generalized} associativity 
\end{enumerate}
In contrast, a $\ZZ$-module (or really an $\NN$-module) $M$ consists of:
\begin{enumerate}[label=(\roman*)]
\item A pointed set $M$ (the ``underlying set'')
\item A family of $n$-ary \emph{operations} $+: M^n \to M$ for each $n \geq 2$ (all reducible to binary $+$)
\item Satisfying \emph{classical} associativity: $(a + b) + c = a + (b + c)$
\end{enumerate}
In the category of abelian groups, all $n$-ary operations reduce to iterated binary operations by associativity, and the associativity law reduces to the 3-element equation $(a + b) + c = a + (b + c)$.
The $\Gamma$-set framework makes explicit the ``higher additive structures'' that are implicit (and trivial) in classical algebra. 
\end{remark}
Therefore, the Eilenberg-MacLane functor $H: \Ab \to \FMod$ can be seen as a \emph{forgetful functor}, forgeting that the hyper-operations are strict (have unique results).

The natural question is: does there exist a \emph{left adjoint} in the opposite direction? That is, can we universally ``strictify'' a family of $n$-ary hyper-operations satisfying generalized associativity into a single binary operation?

The answer is yes, and this left adjoint is the extension of scalars functor $-\otimes_{\FF} \ZZ$. 

\section{Extension of scalars}\label{sec:extension-modules}

We now use generalized associativity to describe morphisms
from $\Gamma$-sets to Eilenberg--MacLane objects and construct
extension of scalars from $\FF$ to $\ZZ$. We first establish
the adjunction for modules, then show that the same construction
inherits a compatible multiplication for commutative
$\FF$-algebras.

\subsection{Maps to Eilenberg--MacLane objects}
\label{subsec:char-maps}

Let $X$ be a $\Gamma$-set and let $M$ be a commutative monoid,
written additively with identity $0$.

\begin{definition}\label{def:pointed-additive}
A map $g:X(1_+)\to M$ is \emph{pointed additive} if
$g(*)=0$ and
\[
c\in a\oplus b
\quad\Longrightarrow\quad
g(c)=g(a)+g(b)
\]
for all $a,b,c\in X(1_+)$.
\end{definition}

Generalized associativity ensures that these binary conditions
also control all higher-arity sums.

\begin{lemma}\label{prop:higher-sums}
Let $g:X(1_+)\to M$ be pointed additive.
For every $n\geq 0$ and every
$c\in\bigoplus_{i=1}^n a_i$, we have
\[
g(c)=\sum_{i=1}^n g(a_i).
\]
\end{lemma}

\begin{proof}
For $n=0$, the assertion is $g(*)=0$, and for $n=1$
it is immediate. The case $n=2$ is the defining condition.

For $n\geq 3$, Theorem~\ref{thm:gen-assoc}, applied to
the partition $\{1\}\cup\{2,\ldots,n\}$, gives
\[
\bigoplus_{i=1}^n a_i
\subseteq
a_1\oplus\left(\bigoplus_{i=2}^n a_i\right).
\]
Thus there exists $d\in\bigoplus_{i=2}^n a_i$ such that
$c\in a_1\oplus d$. Induction yields
\[
g(c)=g(a_1)+g(d)=\sum_{i=1}^n g(a_i).
\]
\end{proof}

\begin{lemma}\label{lem:characterization}
Restriction to the first level gives a natural bijection
between morphisms of $\Gamma$-sets $X\to HM$ and pointed
additive maps $X(1_+)\to M$.
\end{lemma}

\begin{proof}
If $f:X\to HM$ is a morphism, then $g=f_{1_+}$ preserves
the basepoint. By Lemma~\ref{lem:preserve-sums},
$c\in a\oplus b$ implies
\[
g(c)\in g(a)\oplus g(b)=\{g(a)+g(b)\},
\]
so $g$ is pointed additive.

Naturality with respect to the coordinate projections
forces
\[
f_{n_+}(x)
=
\bigl(g(X(p_{1,n})(x)),\ldots,
      g(X(p_{n,n})(x))\bigr).
\]
Hence $f$ is uniquely determined by $g$.

Conversely, let $g:X(1_+)\to M$ be pointed additive and
define $f_{n_+}$ by this formula, with the unique map
at level $0$. These maps preserve basepoints.

To verify naturality, let $\gamma:n_+\to m_+$ be a
pointed map and fix $j\in\{1,\ldots,m\}$. Write
$J=\gamma^{-1}(j)$ and $r=|J|$. Choose a bijection
$J\to\{1,\ldots,r\}$ and extend it to a pointed map
$\rho:n_+\to r_+$ by sending the complement of $J$
to $0$. Then
\[
p_{j,m}\circ\gamma=s_r\circ\rho,
\]
and $X(\rho)(x)$ exhibits
\[
X(p_{j,m}\circ\gamma)(x)
\in
\bigoplus_{i\in J}X(p_{i,n})(x).
\]
This also holds when $J$ is empty, since $X(0_+)=\{*\}$.
Lemma~\ref{prop:higher-sums} therefore gives
\[
g\bigl(X(p_{j,m}\circ\gamma)(x)\bigr)
=
\sum_{i\in J}g\bigl(X(p_{i,n})(x)\bigr).
\]
These are precisely the $j$-th coordinates of
$f_{m_+}(X(\gamma)(x))$ and $HM(\gamma)(f_{n_+}(x))$,
respectively. Thus $f$ is natural.

The correspondence is natural in both $X$ and $M$,
as follows directly from its description by restriction
and the displayed coordinate formula.
\end{proof}

\subsection{The module adjunction}
\label{subsec:def-extension}

\begin{definition}\label{def:extension-modules}
For an $\FF$-module $X$, define
\[
X\otimes_{\FF}\ZZ
:=
\ZZ[X(1_+)]/\mathcal R_X,
\]
where $\ZZ[X(1_+)]$ is the free abelian group on
$X(1_+)$ and $\mathcal R_X$ is the subgroup generated by
\[
[*]
\qquad\text{and}\qquad
[a]+[b]-[c]
\quad\text{for }c\in a\oplus b.
\]
We use $[a]$ for both a generator and its image in the
quotient when no confusion can arise.
\end{definition}

\begin{lemma}\label{lem:functoriality}
The assignment $X\mapsto X\otimes_{\FF}\ZZ$ defines
a functor
\[
-\otimes_{\FF}\ZZ:\FMod\longrightarrow\Ab.
\]
\end{lemma}

\begin{proof}
A morphism $\varphi:X\to Y$ induces the homomorphism
of free abelian groups given by
$[a]\mapsto[\varphi_{1_+}(a)]$.
Since $\varphi$ preserves basepoints and satisfies
\[
c\in a\oplus b
\quad\Longrightarrow\quad
\varphi_{1_+}(c)
\in
\varphi_{1_+}(a)\oplus\varphi_{1_+}(b),
\]
this homomorphism sends $\mathcal R_X$ into
$\mathcal R_Y$ and descends to the quotients.
Compatibility with identities and composition
follows from the formula on generators.
\end{proof}

\begin{theorem}[Extension of scalars for modules]
\label{thm:extension-adjunction-modules}
The functor
\[
-\otimes_{\FF}\ZZ:\FMod\longrightarrow\Ab
\]
is left adjoint to the Eilenberg--MacLane functor
$H:\Ab\to\FMod$. More precisely, there is a natural
bijection
\[
\Hom_{\Ab}(X\otimes_{\FF}\ZZ,M)
\cong
\Hom_{\FMod}(X,HM)
\]
for every $\FF$-module $X$ and abelian group $M$.
\end{theorem}

\begin{proof}
By the universal property of the quotient,
homomorphisms $X\otimes_{\FF}\ZZ\to M$ correspond
bijectively to pointed additive maps $X(1_+)\to M$:
a homomorphism $h$ corresponds to $a\mapsto h([a])$.
By Lemma~\ref{lem:characterization}, these maps
correspond bijectively to morphisms $X\to HM$.
Both correspondences are natural.
\end{proof}

The unit of this adjunction is the morphism
\[
\eta_X:X\longrightarrow H(X\otimes_{\FF}\ZZ)
\]
whose first-level component sends $a$ to $[a]$.
At every level it is given by
\[
(\eta_X)_{n_+}(x)
=
\bigl([X(p_{1,n})(x)],\ldots,
      [X(p_{n,n})(x)]\bigr).
\]
Thus every morphism from $X$ to an Eilenberg--MacLane
object of an abelian group factors uniquely through
$\eta_X$.

\begin{remark}\label{rem:schwede-linearization}
The adjunction above is $\pi_0$ of
Schwede's linearization adjunction restricted to discrete
$\Gamma$-sets \cite{schwede}.
Our construction presents this functor explicitly
in terms of binary hyperaddition relations;
generalized associativity ensures that these
relations also enforce compatibility with all
higher-arity sums.

Let $E(X)$ denote the connective spectrum associated
with a discrete $\Gamma$-set $X$.
By \cite[Lemma~1.2]{schwede}, the unit
\[
X\longrightarrow H(X\otimes_{\FF}\ZZ)
\]
induces an isomorphism on stable $\pi_0$.
The target represents the Eilenberg--MacLane
spectrum of $X\otimes_{\FF}\ZZ$, so naturally
\[
X\otimes_{\FF}\ZZ\cong\pi_0 E(X).
\]

This identifies our construction with the
zeroth homotopy group of spectral extension
of scalars.
Indeed, writing $\mathbb S$ for the sphere
spectrum, extension of scalars along
$\mathbb S\to H\ZZ$ is given by
$E\mapsto E\wedge_{\mathbb S}^{\mathbb L}H\ZZ$.
Since $E(X)$ is connective, the stable
Hurewicz map gives
\[
X\otimes_{\FF}\ZZ
\cong
\pi_0\bigl(E(X)\wedge_{\mathbb S}^{\mathbb L}H\ZZ\bigr).
\]
Thus the algebraic construction captures the
degree-zero part of this spectral base change,
whose higher homotopy groups are the integral
homology groups of $E(X)$.
\end{remark}

\subsection{Examples}\label{subsec:examples-modules}

The following examples recover free abelian groups,
group completion, and quotients.

\begin{example}[Spherical objects]
\label{ex:spherical-extension}
\label{ex:f1-extension}
For any pointed set $Y$, there is a natural isomorphism
\[
(\mathbb SY)\otimes_{\FF}\ZZ
\cong \ZZ[Y\setminus\{*\}].
\]
Indeed, the only binary sums in $\mathbb SY$ involve
the basepoint, so their relations impose no further
constraints once $[*]=0$.
In particular, taking $Y=\{0,1\}$ gives
\[
\FF\otimes_{\FF}\ZZ\cong\ZZ.
\]
\end{example}

\begin{example}[Eilenberg--MacLane objects]
\label{ex:EM-extension}
For a commutative monoid $M$, there is a natural
isomorphism
\[
(HM)\otimes_{\FF}\ZZ\cong M^{\mathrm{gp}},
\]
where $M^{\mathrm{gp}}$ denotes its group completion.
Indeed, binary hypersums in $HM$ are the singletons
$\{a+b\}$, so the defining relations are precisely
\[
[0]=0,\qquad [a]+[b]=[a+b].
\]
These present the universal abelian group receiving
a monoid homomorphism from $M$.

In particular, if $M$ is already an abelian group,
the map $[a]\mapsto a$ gives
\[
(HM)\otimes_{\FF}\ZZ\cong M.
\]
Thus the counit is an isomorphism, and the adjunction
exhibits $\Ab$ as a reflective full subcategory
of $\FMod$.
\end{example}

\begin{example}[Quotients of Eilenberg--MacLane objects]
\label{ex:quotient-extension}
Let $X=H\ZZ/H(n\ZZ)$ for $n\geq2$, with the
quotient formed levelwise in pointed sets by collapsing it to the basepoint. Since extension of scalars is a left adjoint,
it preserves colimits, thus
\[
X\otimes_{\FF}\ZZ
\cong \ZZ/n\ZZ.
\]
\end{example}

\subsection{Multiplicative structure and the algebra adjunction}
\label{sec:algebras}
\label{subsec:extension-alg}

We now extend the module construction to commutative
$\FF$-algebras by showing that multiplication descends
through the hyperaddition relations.

\begin{definition}\label{def:f1-algebra}
A \emph{commutative $\FF$-algebra} is a commutative
monoid object in $(\GSet,\wedge,\FF)$.
Morphisms preserve multiplication and unit.
We denote this category by $\FAlg$ and use
``$\FF$-algebra'' to mean ``commutative
$\FF$-algebra'' throughout.
\end{definition}

\begin{remark}\label{remark:5.3}
By the universal property of Day convolution,
a multiplication $A\wedge A\to A$ is equivalently
a family of maps
\[
\mu_{m,n}:A(m_+)\wedge A(n_+)
\longrightarrow A(m_+\wedge n_+)
\]
natural in both variables.
The algebra axioms translate into associativity,
unit, and symmetry conditions on these maps.
In particular, $A(1_+)$ is a commutative monoid
with absorbing basepoint.
We write its product as $ab$ and its unit as $1_A$.
\end{remark}

\begin{example}[Spherical monoid algebras]
\label{ex:spherical-algebra}
For a commutative monoid $P$ with absorbing
basepoint $0_P$, the spherical object $\mathbb SP$
is an $\FF$-algebra with multiplication
\[
(P\wedge m_+)\wedge(P\wedge n_+)
\longrightarrow P\wedge(m_+\wedge n_+),
\qquad
(p,i)\wedge(q,j)\longmapsto(pq,(i,j)),
\]
and unit determined by $1_P$.
This defines a fully faithful functor
\[
\mathbb S:\CMonoid_*\longrightarrow\FAlg,
\]
left adjoint to evaluation at $1_+$
\cite{connes2019absolute}.
\end{example}

\begin{example}[Eilenberg--MacLane algebras]
\label{ex:EM-algebra}
For a commutative semiring $R$, the object $HR$
is an $\FF$-algebra with multiplication
\[
HR(m_+)\wedge HR(n_+)\longrightarrow HR((mn)_+),
\qquad
(x_i)_i\wedge(y_j)_j\longmapsto(x_i y_j)_{i,j},
\]
and unit determined by $1_R$.
This defines a fully faithful functor
$H:\CSemiring\to\FAlg$, which restricts to
$H:\CRing\to\FAlg$ \cite{connes2019absolute}.
\end{example}

\begin{lemma}\label{lem:multiplication-hypersums}
Multiplication in an $\FF$-algebra preserves
binary hypersums: for $a,b,c,t\in A(1_+)$,
\[
c\in a\oplus b
\quad\Longrightarrow\quad
tc\in ta\oplus tb.
\]
\end{lemma}

\begin{proof}
Choose $z\in A(2_+)$ exhibiting $c\in a\oplus b$.
Under $1_+\wedge2_+\cong2_+$, the element
$w=\mu_{1,2}(t,z)$ satisfies
\[
A(p_1)(w)=ta,\qquad
A(p_2)(w)=tb,\qquad
A(s_2)(w)=tc
\]
by naturality of multiplication.
Thus $w$ exhibits the required relation.
\end{proof}

\begin{proposition}\label{def:extension-algebras}
\label{lem:functoriality-algebra}
For an $\FF$-algebra $A$, the formulas
\[
[a][b]=[ab],
\qquad
1=[1_A]
\]
give $A\otimes_{\FF}\ZZ$ a natural commutative
ring structure.
This defines a functor
\[
-\otimes_{\FF}\ZZ:\FAlg\longrightarrow\CRing
\]
whose underlying abelian group is the extension
of scalars of the underlying $\FF$-module.
\end{proposition}

\begin{proof}
Equip $\ZZ[A(1_+)]$ with the monoid ring
multiplication $[a][b]=[ab]$.
The relation subgroup $\mathcal R_A$ of
Definition~\ref{def:extension-modules} is an ideal:
for every $t\in A(1_+)$,
\[
[t][*]=[*]\in\mathcal R_A,
\]
and, whenever $c\in a\oplus b$,
\[
[t]\bigl([a]+[b]-[c]\bigr)
=
[ta]+[tb]-[tc]\in\mathcal R_A
\]
by Lemma~\ref{lem:multiplication-hypersums}.
Hence the module quotient inherits the stated
ring structure.

For an algebra morphism $\varphi:A\to B$,
the induced additive map
$[a]\mapsto[\varphi_{1_+}(a)]$
preserves multiplication and unit.
Functoriality follows from the module construction.
\end{proof}

\begin{lemma}\label{lem:characterization-algebra}
Let $A$ be an $\FF$-algebra and $R$ a commutative
ring.
Restriction to the first level identifies algebra
morphisms $A\to HR$ with pointed additive maps
$g:A(1_+)\to R$ satisfying
\[
g(ab)=g(a)g(b),
\qquad
g(1_A)=1_R.
\]
\end{lemma}

\begin{proof}
Every algebra morphism gives such a map.
Conversely, Lemma~\ref{lem:characterization}
extends $g$ uniquely to a morphism of
$\Gamma$-sets $f:A\to HR$.
The condition $g(1_A)=1_R$ ensures that $f$
preserves the unit, since maps $\FF\to HR$
are determined at level $1$.

To check multiplication, let $x\in A(m_+)$
and $y\in A(n_+)$.
Naturality of $\mu$ gives
\[
A(p_{i,j})(\mu_{m,n}(x,y))
=
A(p_{i,m})(x)\,A(p_{j,n})(y),
\]
where $p_{i,j}:(mn)_+\to1_+$ projects onto
the coordinate $(i,j)$.
Applying $g$ and using its multiplicativity
shows that the $(i,j)$-th coordinates of
\[
f_{(mn)_+}(\mu_{m,n}(x,y))
\quad\text{and}\quad
\mu^{HR}_{m,n}(f_{m_+}(x),f_{n_+}(y))
\]
agree.
Thus $f$ preserves multiplication.
\end{proof}

\begin{theorem}[Extension of scalars for algebras]
\label{thm:extension-adjunction-algebras}
There is an adjunction
\[
-\otimes_{\FF}\ZZ:\FAlg
\rightleftarrows\CRing:H,
\qquad
(-\otimes_{\FF}\ZZ)\dashv H.
\]
Equivalently, there are natural bijections
\[
\Hom_{\CRing}(A\otimes_{\FF}\ZZ,R)
\cong
\Hom_{\FAlg}(A,HR).
\]
\end{theorem}

\begin{proof}
By the quotient ring presentation,
ring homomorphisms $A\otimes_{\FF}\ZZ\to R$
correspond to pointed additive maps
$A(1_+)\to R$ preserving multiplication and unit.
By Lemma~\ref{lem:characterization-algebra},
these correspond to algebra morphisms $A\to HR$.
Both correspondences are natural.
\end{proof}

\begin{remark}
    The above fits into the following adjunction diagram:

\begin{center}
\begin{tikzcd}
\CMonoid_* \arrow[r,"\mathbb{S}"] \arrow[rr, bend left, "\text{free}"]  &  \FAlg \arrow[r,"-\otimes_{\mathbb{F}_1} \mathbb{Z}"] & \CRing  \\
\CMonoid_*    &  \FAlg \arrow[l,"-(1_+)"]  & \CRing \arrow[l,"H"] \arrow[ll, bend left, "\text{forget}"]
\end{tikzcd}
\end{center}
where the horizontal arrows form adjoint pairs: $\mathbb{S} \dashv -(1_+)$ (Example \ref{ex:spherical-algebra}) and $(-\otimes_{\FF}\ZZ) \dashv H$ (Theorem \ref{thm:extension-adjunction-algebras}). The bent arrows represent the composite functors, with the top path providing the contracted monoid ring on a pointed monoid, and the bottom path extracting the underlying multiplicative monoid, again an adjoint pair.

In \cite{borger2009lambdaringsfieldelement}, Borger interprets $\mathbb{F}_1$-algebras as $\lambda$-rings (rings equipped with operations $\lambda^n$ satisfying axioms from exterior powers), with the Witt vector functor providing a right adjoint to the forgetful functor. His framework can be summarized by the following diagram of adjunctions:

\begin{center}
\begin{tikzcd}
\CMonoid_* \arrow[r, "\ZZ"] \arrow[rr, bend left, "\text{free}"]  &  \lambda\text{-}\mathbf{Ring} \arrow[r, "\text{forget}"] & \CRing  \\
\CMonoid_*    &  \lambda\text{-}\mathbf{Ring}  \arrow[l,"\text{rank }1"]  & \CRing \arrow[l,"W"] \arrow[ll, bend left, "\text{forget}"]
\end{tikzcd}
\end{center}
where on the top side, $\ZZ$ associates a pointed monoid $M$ with its contracted monoid ring naturally equipped with a $\lambda$-ring structure. On the bottom side, the cofree $\lambda$-ring is given by the ring of Witt vectors $W(R)$, whose rank $1$ elements (i.e., elements $x$ with $\lambda^n(x)=0$ for all $n > 1$) naturally form a pointed monoid structure. As in our framework, the horizontal arrows form adjoint pairs.

Both frameworks provide adjunctions relating $\mathbb{F}_1$-structures to rings, addressing different aspects of the ``geometry over $\mathbb{F}_1$''.
\end{remark}

\section{Coherent plasmas and truncation}
\label{sec:coherent-plasmas}

The hyper-operations of a $\Gamma$-set satisfy more than
the partitionwise inclusions of
generalized associativity.
An element exhibiting a hypersum simultaneously determines
all its partial sums, and these partial sums are compatible
with every grouping into disjoint subsets.
This simultaneous compatibility is the condition needed
to have a fully faithful embedding of the hyper-additive structure into the category of $\Gamma$-sets.

In this section, we formulate this condition intrinsically.
We obtain categories $\mathsf{Plas}_n$ of structures with
hyper-operations of arity at most $n$, with
$\mathsf{Plas}_2$ equal to the category of plasmas of
Beardsley and Nakamura \cite{beardsley2024}.
At every finite stage, there is a fully faithful nerve
into truncated $\Gamma$-sets, and its right Kan extension gives a fully faithful embedding
into $\FMod$.
Allowing all finite arities gives the category of
coherent plasmas.

\subsection{Coherent hyper-additive structures}
\label{subsec:coherent-hyperoperations}

The proof of Theorem~\ref{thm:gen-assoc} contains
a compatibility stronger than mere
partitionwise inclusions.
Indeed, suppose that $\xi\in X(k_+)$ exhibits
\[
c\in\bigoplus_{i=1}^k a_i.
\]
For every subset $S\subseteq[k]$, let
$\rho_S:k_+\to1_+$ be the pointed map taking
the value $1$ exactly on $S$, and set
\[
a_S=X(\rho_S)(\xi).
\]
The same element $\xi$ therefore determines
all partial sums simultaneously, in particular:
\[
a_\varnothing=0,\qquad
a_{\{i\}}=a_i,\qquad
a_{[k]}=c.
\]
Moreover, these partial sums are compatible
with every grouping into disjoint subsets.
The proof of generalized associativity uses
this construction for one partition at a time;
we now retain its simultaneous compatibility
as an intrinsic axiom.

\begin{definition}\label{def:coherent-plasma}
A \emph{coherent plasma}, or an
\emph{$\infty$-plasma}, consists of a pointed
set $(M,0)$ equipped with hyper-operations
\[
\mu_k:M^k\longrightarrow\mathcal P(M),
\qquad k\geq0,
\]
satisfying the following conditions.
\begin{enumerate}[label=(\roman*)]
\item \emph{Normalization:}
\[
\mu_0=\{0\},
\qquad
\mu_1(a)=\{a\}.
\]

\item \emph{Coherent subset sums:}
For every $k\geq0$ and every
$c\in\mu_k(a_1,\ldots,a_k)$, there exists
a family
\[
(a_S)_{S\subseteq[k]}
\]
such that
\[
a_\varnothing=0,\qquad
a_{\{i\}}=a_i,\qquad
a_{[k]}=c,
\]
and, for every $r\geq1$ and every ordered
family of pairwise disjoint subsets
$S_1,\ldots,S_r\subseteq[k]$,
\begin{equation*}\label{eq:coherent-subset-sums}
a_{S_1\cup\cdots\cup S_r}
\in
\mu_r(a_{S_1},\ldots,a_{S_r}).
\end{equation*}
Empty subsets are allowed.
\end{enumerate}

A morphism $f:M\to M'$ of coherent plasmas
is a pointed map satisfying
\[
f\bigl(\mu_k(a_1,\ldots,a_k)\bigr)
\subseteq
\mu'_k(f(a_1),\ldots,f(a_k))
\qquad(k\geq0).
\]
We denote this category by
$\mathsf{Plas}_\infty$.
\end{definition}

\begin{example}\label{prop:gamma-coherent-plasma}
Let $X$ be a $\Gamma$-set. Then $X(1_+)$ together with its family of $n$-ary hyper-operations
\begin{align*}
    \mu_k:X(1_+)^k&\rightarrow\mathcal P(X(1_+))\\
    (a_1,a_2,...,a_k) &\mapsto \bigoplus_{i=1}^k a_i
\end{align*}
form a coherent plasma. A morphism of $\Gamma$-sets induces a morphism of coherent plasmas.

We denote this coherent plasma by $T_\infty X$.
\end{example}

\begin{remark}
Generalized associativity expresses compatibility with
each partition through an inclusion of hypersums.
The coherent subset-sum axiom requires a single family
$(a_S)_{S\subseteq[k]}$ satisfying all disjoint-grouping
relations simultaneously, including those involving
overlapping subsets in different groupings.
\end{remark}

\begin{example}[The powerset plasma]
\label{ex:powerset-plasma}
For each integer $n\geq0$, let
\[
 P_n=\mathcal P([n])
\]
the powerset of $n$ element set, with basepoint $\varnothing$.
Define its hyper-operations by
\[
\mu_k(S_1,\ldots,S_k)
=
\begin{cases}
\{S_1\cup\cdots\cup S_k\},
& \text{if the }S_i\text{ are pairwise disjoint},\\
\varnothing,
& \text{otherwise},
\end{cases}
\]

These operations make $P_n$ a
coherent plasma.
Indeed, for a permitted sum of pairwise
disjoint subsets $S_1,\ldots,S_k$, one can define
\[
S_A=\bigcup_{i\in A}S_i
\qquad(A\subseteq[k]).
\]
making it the required
coherent family.
\end{example}

The powerset plasmas assemble
into a functor
\[
P:\Gamma\longrightarrow
\mathsf{Plas}_\infty,
\qquad
n_+\longmapsto P_n,
\]
which sends the opposite of a pointed map
$\alpha:n_+\to m_+$ to
\[
\alpha^*: P_m\longrightarrow P_n,
\qquad
T\longmapsto\alpha^{-1}(T).
\]

\begin{definition}\label{def:coherent-plasmic-nerve}
The \emph{nerve} of a coherent plasma $M$ is
the $\Gamma$-set defined by
\[
N_\infty M
=
\Hom_{\mathsf{Plas}_\infty}(P(-),M),
\]
pointed by the constant zero morphism. Since $P_0$ consists only of its basepoint, $N_\infty M(0_+)=\{*\}$ and this is a well defined $\Gamma$-set.
\end{definition}

Let us unwind this definition. A morphism from $P_n$ to $M$ is a family of elements $(a_S)_{S\subseteq[n]}$ such that
\[
a_{S_1\cup\cdots\cup S_r}
\in\mu_r^M(a_{S_1},\ldots,a_{S_r})
\]
for every pairwise disjoint $S_1,\ldots,S_r\subseteq[n]$, and that $a_\varnothing=0$.

A pointed map
$\alpha:n_+\to m_+$ sends $(a_S)_{S\subseteq[n]}$ to 
\[
\bigl(a_{\alpha^{-1}(T)}\bigr)_{T\subseteq[m]}
\]

A morphism $f:M\to M'$ acts componentwise:
\[
(a_S)_S\longmapsto(f(a_S))_S.
\]

Recall by Example~\ref{prop:gamma-coherent-plasma},
extracting the hyper-operations of a $\Gamma$-set
defines a functor
\[
T_\infty:\GSet\longrightarrow\mathsf{Plas}_\infty.
\]
with its underlying pointed set being $X(1_+)$. We now show that $T_\infty$ and $N_\infty$ forms an adjoint pair:

\begin{theorem}\label{prop:nerve-recovers-operations}
The extraction functor is left adjoint to
the nerve functor:
\[
T_\infty:\GSet
\rightleftarrows
\mathsf{Plas}_\infty:N_\infty,
\qquad
T_\infty\dashv N_\infty.
\]
Moreover, $T_\infty$ is a left inverse to
$N_\infty$ up to natural isomorphism,
and $N_\infty$ is fully faithful.
\end{theorem}

\begin{proof}
We first show that $T_\infty N_\infty \cong id$. Let $M$ be a coherent plasma, then every $a\in M$ determines
a unique morphism $P_1\to M$.
Thus $N_\infty M(1_+) \cong M$ as a pointed set.

Under this identification, let
$c\in\mu_n^{T_\infty N_\infty M}(a_1,\ldots,a_n)$.
This corresponds to a morphism $\varphi:P_n\to M$
such that
\[
(N_\infty M)(p_i)\varphi=a_i,
\qquad
(N_\infty M)(s)\varphi=c.
\]
which is equivalent to 
\[
\varphi(\{i\})=a_i,
\qquad
\varphi([n])=c.
\]
Such a morphism implies
$c\in\mu_n^M(a_1,\ldots,a_n)$, and the converse
is precisely the coherent subset-sum axiom which gives a morphism from $P_n$ to $M$.
Hence, naturally in $M$,
\[
T_\infty N_\infty M\cong M.
\]

Now, given a morphism $f:T_\infty X\to M$, define
$\widehat f:X\to N_\infty M$ levelwise by
\[
\widehat f_n(\xi)(S)
=
f\bigl(X(\rho_S)(\xi)\bigr),
\qquad S\subseteq[n].
\]
where $\rho_S:n_+\to1_+$ is the pointed map taking
the value $1$ exactly on $S$.

The subset sums of $\xi$ form a coherent
family, so this defines a morphism $P_n\to M$.
It is straightforward to verify that $\widehat f$ is natural, and that it preserves basepoints.

Conversely, a natural transformation
$\eta:X\to N_\infty M$ induces
\[
T_\infty\eta:T_\infty X
\longrightarrow T_\infty N_\infty M\cong M.
\]
with underlying map being $f=\eta_{1_+}$.
Naturality with respect to $\rho_S$ forces
\[
\eta_n(\xi)(S)
=
f\bigl(X(\rho_S)(\xi)\bigr).
\]
The two constructions are therefore inverse
and natural in $X$ and $M$, proving
$T_\infty\dashv N_\infty$.
Its counit is the isomorphism above, so
$N_\infty$ is fully faithful.
\end{proof}

\subsection{Coherent Segal condition}
\label{subsec:essential-image}

By Theorem~\ref{prop:nerve-recovers-operations}, coherent plasmas form a
reflective full subcategory of $\GSet$. We characterize this subcategory
by a condition stated entirely in terms of a $\Gamma$-set: every coherent family of subset sums must admit a unique simultaneous
realization.

\begin{definition}\label{def:subset-sum-map}
Let $X$ be a $\Gamma$-set. For $S\subseteq[n]$, let
$\rho_S\colon n_+\to1_+$ be the pointed map taking the value $1$
exactly on $S$. The \emph{$n$-th subset-sum map} is
\[
\sigma_n\colon X(n_+)\longrightarrow
\prod_{S\subseteq[n]}X(1_+),
\qquad
\xi\longmapsto
\bigl(X(\rho_S)(\xi)\bigr)_{S\subseteq[n]}.
\]
We call $X$ \emph{subset-separated} if every $\sigma_n$ is injective.
\end{definition}

Restricting $\sigma_n$ to singleton subsets gives the usual Segal map
\[
\mathrm{seg}_n\colon X(n_+)\longrightarrow X(1_+)^n,
\qquad
\xi\longmapsto
\bigl(X(p_{1,n})(\xi),\ldots,X(p_{n,n})(\xi)\bigr).
\]
The remaining components record the other subset sums.

\begin{definition}\label{def:locally-realizable}
A family $(a_S)_{S\subseteq[n]}$ in $X(1_+)$ is a 
\emph{coherent family} if $a_\varnothing=*$ and, for every
$r\geq1$ and every collection of pairwise disjoint subsets
$S_1,\ldots,S_r\subseteq[n]$, there exists $\eta\in X(r_+)$ such that
\[
X(p_{j,r})(\eta)=a_{S_j}\quad(1\leq j\leq r),
\qquad
X(s_r)(\eta)=a_{S_1\cup\cdots\cup S_r}.
\]
Write $C_n(X)$ for the set of such families.
\end{definition}

The local witnesses may be chosen independently for different
collections of subsets. The following lemma relates these families
to the nerve constructed above.

\begin{lemma}\label{lem:unit-is-subset-sum}
There is a canonical identification
\[
C_n(X)=N_\infty T_\infty X(n_+)
\]
under which the unit of $T_\infty\dashv N_\infty$ is
\[
(u_X)_{n_+}=\sigma_n\colon X(n_+)\longrightarrow C_n(X).
\]
\end{lemma}

\begin{proof}
By the definition of the hypersums of $X$, coherence says
exactly that $a_\varnothing=*$ and
\[
a_{S_1\cup\cdots\cup S_r}
\in\bigoplus_{j=1}^r a_{S_j}
\]
for every pairwise disjoint collection $S_1,\ldots,S_r$.
These are precisely the coherent families defining
$N_\infty T_\infty X(n_+)$.

For $\xi\in X(n_+)$, its subset-sum family is coherent:
the pointed map $\gamma\colon n_+\to r_+$ sending $S_j$ to $j$
and every other element to $0$ makes $X(\gamma)(\xi)$ a local
witness. Finally, applying the adjunction formula from
Theorem~\ref{prop:nerve-recovers-operations} to
$\mathrm{id}_{T_\infty X}$ gives
\[
(u_X)_{n_+}(\xi)(S)=X(\rho_S)(\xi).
\]
\end{proof}

\begin{definition}\label{def:coherent-segal}
A $\Gamma$-set $X$ satisfies the \emph{coherent Segal condition} if
\[
\sigma_n\colon X(n_+)\longrightarrow C_n(X)
\]
is bijective for every $n\geq0$.
\end{definition}

Thus every coherent family must have a simultaneous
realization, and subset-separation requires this realization
to be unique.

\begin{theorem}[Essential image of the nerve]
\label{thm:essential-image}
For a $\Gamma$-set $X$, the following are equivalent:
\begin{enumerate}[label=(\arabic*)]
\item $X$ lies in the essential image of
$N_\infty\colon\mathsf{Plas}_\infty\to\GSet$;
\item the unit $u_X\colon X\to N_\infty T_\infty X$ is an isomorphism;
\item $X$ satisfies the coherent Segal condition.
\end{enumerate}
Consequently, $T_\infty$ and $N_\infty$ restrict to mutually inverse
equivalences between $\mathsf{Plas}_\infty$ and the full subcategory
of $\GSet$ consisting of the $\Gamma$-sets satisfying the coherent
Segal condition.
\end{theorem}

\begin{proof}
Since the counit is an isomorphism by
Theorem~\ref{prop:nerve-recovers-operations}, the essential image
of $N_\infty$ consists precisely of those objects whose unit is
an isomorphism. This proves (1)$\Leftrightarrow$(2).
The equivalence (2)$\Leftrightarrow$(3) follows from
Lemma~\ref{lem:unit-is-subset-sum}, and the final statement follows
by restricting the adjunction.
\end{proof}

\begin{remark}\label{rem:segal-hierarchy}
Theorem~\ref{thm:essential-image} places coherent plasmas in a hierarchy of
conditions on the levels of a $\Gamma$-set:
\begin{center}
\renewcommand{\arraystretch}{1.25}
\begin{tabular}{ll}
\textbf{condition on $X(n_+)$} & \textbf{corresponding structure}\\ \hline
$\mathrm{seg}_n\colon X(n_+)\to X(1_+)^n$ bijective & commutative monoid\\
$\sigma_n\colon X(n_+)\to N_\infty T_\infty X(n_+)$ bijective & coherent plasma\\
none & general $\Gamma$-set
\end{tabular}
\end{center}
Under the Segal condition, an element of $X(n_+)$ is determined by its $n$
coordinates, and every $n$-tuple of coordinates occurs. Under the coherent Segal
condition, it is determined by all $2^n$ of its partial sums, and every coherent
family of partial sums occurs.
The following proposition shows that the first condition
implies the second.
\end{remark}

\begin{proposition}\label{prop:segal-implies-coherent}
If every Segal map $\mathrm{seg}_n$ of $X$ is bijective, then
$X$ satisfies the coherent Segal condition.
\end{proposition}

\begin{proof}
Since $\mathrm{seg}_n$ factors through $\sigma_n$, the latter
is injective. Moreover, bijectivity of every $\mathrm{seg}_k$
makes every hypersum a singleton.

Given $(a_S)\in C_n(X)$, let $\xi\in X(n_+)$ be the unique
element with coordinates $a_{\{1\}},\ldots,a_{\{n\}}$.
For every nonempty $S\subseteq[n]$, both $a_S$ and
$X(\rho_S)(\xi)$ belong to the singleton
\[
\bigoplus_{i\in S}a_{\{i\}}.
\]
Their empty components are both $*$, so
$\sigma_n(\xi)=(a_S)$. Thus $\sigma_n$ is also surjective.
\end{proof}

\begin{remark}\label{rem:coherent-examples}
In particular, $HM$ satisfies the coherent Segal condition for
every commutative monoid $M$.

The spherical objects also satisfy this
condition. Indeed, a coherent family for
$T_\infty\mathbb SY$ has at most one nonzero singleton value,
and all its other values are then uniquely determined
by Example~\ref{ex:spherical-sums}. Hence
\[
N_\infty T_\infty\mathbb SY\cong\mathbb SY.
\]

Since $H$ and $\mathbb S$ are fully faithful and land in the
essential image of $N_\infty$, Theorem~\ref{thm:essential-image}
also gives fully faithful embeddings
\[
T_\infty H\colon\CMonoid\longrightarrow\mathsf{Plas}_\infty,
\qquad
T_\infty\mathbb S\colon\Set_*\longrightarrow\mathsf{Plas}_\infty.
\]
\end{remark}

\begin{remark}[Commutative hypergroups]
Let $\widehat H$ denote the Beardsley--Nakamura plasmic nerve,
which is fully faithful by \cite[Corollary~3.15]{beardsley2024}. Its image satisfies the coherent Segal condition, since
coherent families satisfy the binary
disjoint-union relations defining the plasmic nerve. Thus there is also a fully faithful embedding of commutative
hypergroups into the category of coherent plasmas
\[
T_\infty\widehat H\colon
\mathsf{CHyp}\longrightarrow\mathsf{Plas}_\infty
\]
\end{remark}

\begin{example}[A $\Gamma$-set that is not subset-separated]
\label{ex:not-separated}
Define
\[
X(n_+)=\bigwedge\nolimits^2\ZZ^n,
\qquad
X(f)=\bigwedge\nolimits^2 H\ZZ(f),
\]
with basepoint zero. Functoriality of the exterior square
makes $X$ a $\Gamma$-set. Since
\[
X(1_+)=0,
\qquad
X(2_+)\cong\ZZ,
\]
all subset sums of every element of $X(2_+)$ are zero.
Thus $\sigma_2$ is not injective, and $X$ fails the
coherent Segal condition.
This illustrates that a $\Gamma$-set may contain
higher-level information invisible to all subset sums.
\end{example}

\begin{example}[Failure of surjectivity]
\label{ex:coherent-segal-surjectivity}
Let
\[
X(n_+)=\ZZ^n/\{\pm1\},
\]
where the sign acts simultaneously on all coordinates.
The structure maps of $H\ZZ$ commute with this action,
so $X$ is a $\Gamma$-set.
Identify $X(1_+)$ with $\NN$ by absolute value. Consider the following family indexed by subsets of $[3]$:
\[
a_\varnothing=0,\qquad
a_{\{i\}}=1,\qquad
a_{\{i,j\}}=0\quad(i\ne j),\qquad
a_{[3]}=1.
\]
This family satisfies all the coherent subset-sum
conditions.
Indeed, after omitting empty subsets and unary
conditions, the only relations to check are
\[
0\in 1\oplus1,\qquad
1\in 0\oplus1,\qquad
1\in \mu_3(1,1,1).
\]
They are witnessed respectively by
$(1,-1)$, $(0,1)$, and $(1,1,-1)$. Nevertheless, the family is not realized by an element
of $X(3_+)$.
Such a realization would require
\[
|x_1|=|x_2|=|x_3|=1,
\qquad
x_1+x_2=x_1+x_3=x_2+x_3=0,
\]
which is impossible.
Thus
\[
\sigma_3:X(3_+)\longrightarrow C_3(X)
\]
is not surjective: the grouping relations can each
be witnessed, but no single element realizes them
all simultaneously.
\end{example}

\subsection{Truncated $\Gamma$-sets and truncated nerve}
\label{subsec:truncated-plasmic-nerve}

\begin{definition}
    Let $\mathcal C=\Gamma^{\mathrm{op}}$.
For finite $n\geq1$, let $\mathcal C_{\leq n}$
be its full subcategory on
$0_+,1_+,\ldots,n_+$, and write
\[
\mathcal G_n
=
\left\{
X:\mathcal C_{\leq n}\to\Set_*
\ \middle|\ X(0_+)=\{*\}
\right\}.
\]
Morphisms in $\mathcal G_n$ are pointed natural
transformations.
We set $\mathcal C_{\leq\infty}=\mathcal C$ and
$\mathcal G_\infty=\GSet=\FMod$.

We call such a pointed functor from $\mathcal C_{\leq n}$ to $\Set_*$ an $n$-truncated $\Gamma$-set.
\end{definition}

Let $X$ be a $\Gamma$-set. The natural inclusion functor $i_n:\mathcal C_{\leq n}\hookrightarrow\mathcal C$ gives an $n$-truncated $\Gamma$-set $X \circ i_n$ for every $n$, extracting its information on the first $n$ levels. Therefore, this corresponds naturally to a pointed set with hyper-operations up to level $n$:

\begin{definition}\label{def:truncated-plasma}
Let $n\geq1$. An \emph{$n$-plasma} is a pointed
set $(M,0)$ equipped with hyper-operations
\[
\mu_k:M^k\longrightarrow\mathcal P(M),
\qquad 0\leq k\leq n,
\]
satisfying normalization and the coherent
subset-sum axiom of
Definition~\ref{def:coherent-plasma},
with $k\leq n$ and $1\leq r\leq n$.

Morphisms are pointed maps preserving all
these hyper-operations.
We denote the resulting category by
$\mathsf{Plas}_n$.
\end{definition}

Thus every permitted sum of at most $n$ inputs
admits a simultaneous choice of all its
subset sums, compatible with every grouping
into at most $n$ pairwise disjoint subsets.
Empty subsets are again allowed.

For $n=1$, there are no conditions beyond
normalization, so $\mathsf{Plas}_1\cong\Set_*$.
For $n=2$, the axioms say precisely that
$\mu_2$ is commutative and that $0$ is a
weak unit:
\[
a\in\mu_2(a,0)=\mu_2(0,a).
\]
Indeed, the weak unit condition follows by
applying coherence to $\mu_1(a)=\{a\}$,
and commutativity follows by interchanging
the two singleton subsets.
Conversely, these two properties ensure that
every relation $c\in\mu_2(a,b)$ determines
a coherent family by setting
\[
a_\varnothing=0,\qquad
a_{\{1\}}=a,\qquad
a_{\{2\}}=b,\qquad
a_{\{1,2\}}=c.
\]
Hence $\mathsf{Plas}_2$ is the category of
plasmas of Beardsley and Nakamura
\cite{beardsley2024}.

\begin{definition}
    For $X\in\mathcal G_n$, define $T_nX$
on the pointed set $X(1_+)$ by
\[
\mu_k(a_1,\ldots,a_k)
=
\left\{
X(s_k)(\xi)
\ \middle|\
\xi\in X(k_+),\
X(p_i)(\xi)=a_i\text{ for all }i
\right\},
\qquad 1\leq k\leq n,
\]
and set $\mu_0=\{0\}$.
Here $p_i:k_+\to1_+$ is the $i$th
projection and $s_k:k_+\to1_+$ is the
fold map.
\end{definition}

The subset-sum construction used above
takes place entirely in
$\mathcal C_{\leq n}$:
for $k,r\leq n$, a family of pairwise
disjoint subsets of $[k]$ determines a
pointed map $k_+\to r_+$.
Consequently, $T_nX$ is an $n$-plasma,
and extraction defines a functor
\[
T_n:\mathcal G_n\longrightarrow
\mathsf{Plas}_n.
\]

For $0\leq q\leq n$, let $P_q^{(n)}$
denote the powerset plasma $P_q$ with only
its operations of arity at most $n$ retained.
Inverse images define a functor
\[
P^{(n)}:\mathcal C_{\leq n}^{\mathrm{op}}
\longrightarrow\mathsf{Plas}_n.
\]

\begin{definition}\label{def:truncated-nerve}
The \emph{$n$-truncated nerve} of an
$n$-plasma $M$ is
\[
N_nM(q_+)
=
\Hom_{\mathsf{Plas}_n}(P_q^{(n)},M),
\qquad 0\leq q\leq n.
\]
For a pointed map $\alpha:q_+\to p_+$,
its action is given by precomposition
with $\alpha^*$.
\end{definition}

Explicitly, an element of $N_nM(q_+)$
is a family $(a_S)_{S\subseteq[q]}$
with $a_\varnothing=0$ such that
\[
a_{S_1\cup\cdots\cup S_r}
\in\mu_r^M(a_{S_1},\ldots,a_{S_r})
\]
whenever $1\leq r\leq n$ and the subsets
$S_1,\ldots,S_r\subseteq[q]$ are pairwise
disjoint.
The action of $\alpha:q_+\to p_+$ is
\[
(a_S)_{S\subseteq[q]}
\longmapsto
(a_{\alpha^{-1}(T)})_{T\subseteq[p]}.
\]
The constant zero families give the
basepoints, and $N_nM(0_+)=\{*\}$.
Thus $N_nM$ belongs to $\mathcal G_n$.

\begin{theorem}\label{thm:truncated-nerve}
There is an adjunction
\[
T_n:\mathcal G_n
\rightleftarrows
\mathsf{Plas}_n:N_n,
\qquad T_n\dashv N_n.
\]
Its counit is a natural isomorphism
\[
T_nN_nM\cong M.
\]
In particular, $N_n$ is fully faithful.
\end{theorem}

\begin{proof}
Evaluation at $\{1\}$ identifies
$N_nM(1_+)$ with $M$.
Under this identification, a relation
\[
c\in\mu_k^{T_nN_nM}(a_1,\ldots,a_k)
\]
is witnessed by a morphism
$P_k^{(n)}\to M$ with singleton values
$a_1,\ldots,a_k$ and total value $c$.
Such a morphism implies
$c\in\mu_k^M(a_1,\ldots,a_k)$,
and the converse follows from the
$n$-plasma coherence axiom.
Hence $T_nN_nM\cong M$ naturally.

The rest follows by similar arguments in Theorem \ref{prop:nerve-recovers-operations}.
\end{proof}

For $1\leq m\leq n\leq\infty$, let
\[
r_{m,n}:\mathcal G_n\longrightarrow\mathcal G_m,
\qquad
U_{m,n}:\mathsf{Plas}_n\longrightarrow\mathsf{Plas}_m
\]
denote restriction and forgetting operations of arity
greater than $m$, respectively.
Restriction has a fully faithful right adjoint $R_{m,n}$,
given by right Kan extension. $U_{m,n}$ also has a fully faithful right adjoint $J_{m,n}$, which we give explicitly below:

\begin{proposition}\label{prop:plasma-truncation-adjunction}
For $1\leq m\leq n\leq\infty$, the forgetful functor
\[
U_{m,n}:\mathsf{Plas}_n\longrightarrow\mathsf{Plas}_m
\]
has a fully faithful right adjoint $J_{m,n}$.
Moreover,
\[
N_nJ_{m,n}\cong R_{m,n}N_m,
\qquad
J_{m,n}\cong T_nR_{m,n}N_m.
\]
\end{proposition}

\begin{proof}
For an $m$-plasma $M$, define $J_{m,n}M$ on the same
pointed set. Set $\mu_0^{J_{m,n}M}=\{0\}$, and declare
\[
c\in\mu_k^{J_{m,n}M}(a_1,\ldots,a_k)
\qquad(1\leq k\leq n)
\]
if there exists a family $(a_S)_{S\subseteq[k]}$ with
\[
a_\varnothing=0,\qquad a_{\{i\}}=a_i,\qquad a_{[k]}=c,
\]
satisfying
\[
a_{S_1\cup\cdots\cup S_r}
\in\mu_r^M(a_{S_1},\ldots,a_{S_r})
\]
for every $1\leq r\leq m$ and every pairwise disjoint
collection $S_1,\ldots,S_r\subseteq[k]$.

It is straightforward to check that this is well defined and that there is a natural bijection
\[
\Hom_{\mathsf{Plas}_m}(U_{m,n}A,M)
\cong
\Hom_{\mathsf{Plas}_n}(A,J_{m,n}M).
\]
The counit is the identity, so $J_{m,n}$ is fully faithful.

Finally, taking right adjoints of
$U_{m,n}T_n\cong T_mr_{m,n}$ gives
$N_nJ_{m,n}\cong R_{m,n}N_m$.
Applying $T_n$ yields the second isomorphism.
\end{proof}

\begin{remark}
In particular, ordinary plasmas embed fully faithfully
into coherent plasmas via
\[
J_{2,\infty}:
\mathsf{Plas}_2\hookrightarrow\mathsf{Plas}_\infty.
\]
Under this embedding, the coherent nerve recovers the
Beardsley--Nakamura plasmic nerve:
\[
N_\infty J_{2,\infty}\cong\widehat H.
\]
\end{remark}

These adjunctions fit into the following tower,
illustrated for $2<n<\infty$:
\[
\begin{tikzcd}[row sep=huge, column sep=huge]
\mathcal G_\infty
  \arrow[r, shift left=.7ex, "T_\infty"]
  \arrow[d, shift left=.7ex, "r_{n,\infty}"]
&
\mathsf{Plas}_\infty
  \arrow[l, shift left=.7ex, "N_\infty"]
  \arrow[d, shift left=.7ex, "U_{n,\infty}"]
\\
\mathcal G_n
  \arrow[r, shift left=.7ex, "T_n"]
  \arrow[u, shift left=.7ex, "R_{n,\infty}"]
  \arrow[d, shift left=.7ex, "r_{2,n}"]
&
\mathsf{Plas}_n
  \arrow[l, shift left=.7ex, "N_n"]
  \arrow[u, shift left=.7ex, "J_{n,\infty}"]
  \arrow[d, shift left=.7ex, "U_{2,n}"]
\\
\mathcal G_2
  \arrow[r, shift left=.7ex, "T_2"]
  \arrow[u, shift left=.7ex, "R_{2,n}"]
  \arrow[d, shift left=.7ex, "r_{1,2}"]
&
\mathsf{Plas}_2
  \arrow[l, shift left=.7ex, "N_2"]
  \arrow[u, shift left=.7ex, "J_{2,n}"]
  \arrow[d, shift left=.7ex, "U_{1,2}"]
\\
\mathcal G_1
  \arrow[r, shift left=.7ex, "T_1"]
  \arrow[u, shift left=.7ex, "R_{1,2}"]
&
\mathsf{Plas}_1
  \arrow[l, shift left=.7ex, "N_1"]
  \arrow[u, shift left=.7ex, "J_{1,2}"]
\end{tikzcd}
\]
Here
\[
T_k\dashv N_k,
\qquad
r_{m,n}\dashv R_{m,n},
\qquad
U_{m,n}\dashv J_{m,n},
\]
and all the displayed right adjoints are fully faithful.
The extraction functors commute with restriction,
and the nerves commute with the corresponding right adjoints:
\[
U_{m,n}T_n\cong T_mr_{m,n},
\qquad
N_nJ_{m,n}\cong R_{m,n}N_m.
\]
Moreover, uniqueness of right adjoints gives
\[
J_{\ell,n}\cong J_{m,n}J_{\ell,m}
\qquad(1\leq\ell\leq m\leq n\leq\infty).
\]
Thus the finite-arity plasma categories embed compatibly
into one another and into $\mathsf{Plas}_\infty$.
In particular, $J_{2,\infty}$ realizes ordinary plasmas
as a full subcategory of coherent plasmas.

At the bottom,
$\mathcal G_1\simeq\mathsf{Plas}_1\simeq\Set_*$,
and $T_1\dashv N_1$ is an equivalence.
At each stage $n$, the nerve $N_n$ identifies
$\mathsf{Plas}_n$ with a reflective full subcategory
of $\mathcal G_n$.
An $n$-truncated $\Gamma$-set $X$ belongs to this
subcategory precisely when, for every $q\leq n$,
each coherent subset-sum family in $T_nX$ indexed
by $\mathcal P([q])$ is realized by a unique element
of $X(q_+)$.
Equivalently, the unit
\[
X\longrightarrow N_nT_nX
\]
is an isomorphism.

\begin{remark}[Algebraic realization]
\label{rem:algebraic-realization}
The powerset plasmas admit a useful interpretation
as realizations of representable $\Gamma$-sets.
Let
\[
h_n(k_+)=\Hom_{\Set_*}(n_+,k_+),
\]
pointed by the zero map.
Then there is a natural isomorphism
\[
T_\infty h_n\cong P_n.
\]
Indeed, the adjunction and the Yoneda lemma give,
naturally in $M$,
\[
\begin{aligned}
\Hom_{\mathsf{Plas}_\infty}(T_\infty h_n,M)
&\cong \Hom_{\GSet}(h_n,N_\infty M)\\
&\cong N_\infty M(n_+)\\
&=\Hom_{\mathsf{Plas}_\infty}(P_n,M).
\end{aligned}
\]
Concretely, pointed maps $n_+\to1_+$ correspond
to subsets of $[n]$, while the nonzero fibers
of a pointed map $n_+\to k_+$ are pairwise disjoint.
Thus the extracted hyper-operations on $h_n(1_+)$
are precisely the disjoint-union operations of $P_n$.

This is analogous to geometric realization, which sends
the simplicial set $\Delta[n]$ to the
standard topological $n$-simplex:
\[
|\Delta[n]|\cong\Delta^n_{\mathrm{top}}.
\]
In our setting, $h_n$ plays the role of the representable
simplex, while $P_n$ is its algebraic realization:
a standard configuration of $n$ inputs together with
all their compatible partial sums.
Here $n$ records the arity rather than a geometric dimension.

Just as geometric realization assembles a space from
standard $n$-simplices, $T_\infty$ assembles a coherent
plasma from powerset plasmas $P_n$.
Indeed, every $\Gamma$-set is a colimit of representables,
and the left adjoint $T_\infty$ preserves these colimits.
Thus $T_\infty$ can be viewed as an algebraic realization
functor, with $N_\infty$ recording maps from the
standard summation objects $P_n$.
\end{remark}

We end with the following example that shows that no fixed finite truncation determines all higher operations of a coherent plasma:

\begin{example}[Sums with at most $n$ nonzero inputs]
\label{ex:bounded-arity}
For each $n\geq1$, define a sub-$\Gamma$-set
$X_{\leq n}\subseteq H\NN$ by
\[
X_{\leq n}(k_+)
=
\left\{
(a_1,\ldots,a_k)\in\NN^k
\ \middle|\
\#\{i:a_i\neq0\}\leq n
\right\}.
\]
This is preserved by the structure maps of $H\NN$,
since summing along fibers cannot increase the number
of nonzero coordinates.

The coherent plasma $T_\infty X_{\leq n}$ has underlying
pointed set $\NN$ and operations
\[
\mu_k(a_1,\ldots,a_k)
=
\begin{cases}
\{a_1+\cdots+a_k\},
& \#\{i:a_i\neq0\}\leq n,\\
\varnothing,
& \#\{i:a_i\neq0\}>n.
\end{cases}
\]
Thus at most $n$ nonzero inputs may be summed at once.
These operations agree with those of $T_\infty H\NN$
in every arity $k\leq n$, but
\[
\mu_{n+1}^{T_\infty X_{\leq n}}(1,\ldots,1)
=\varnothing,
\qquad
\mu_{n+1}^{T_\infty H\NN}(1,\ldots,1)
=\{n+1\}.
\]
Hence operations of arity at most $n$ do not determine
the higher operations, even among coherent plasmas.
\end{example}

\vspace{1em}

\bibliographystyle{amsalpha}  % or amsplain, alpha, etc.
\bibliography{ref}

\end{document}